\documentclass[12pt]{article}

\usepackage[left=30mm, right=30mm, top=30mm, bottom=30mm, marginparwidth=20mm]{geometry}
\usepackage{setspace}
\usepackage{amsmath}
\usepackage{amsthm}
\usepackage{amssymb}
\usepackage{amsfonts}
\usepackage{mathtools}
\usepackage{hyperref}
\usepackage[capitalize]{cleveref}
\usepackage[backend=biber,style=alphabetic]{biblatex}
\usepackage{tikz-cd}
\usepackage{mathrsfs}
\usepackage{todonotes}
\usepackage{bbm}
\usepackage{quiver}
\usepackage{xurl}

\DeclareFontShape{OT1}{cmss}{m}{sc}{<->ssub*xcmss/m/sc}{}

\newtheorem{definition}{Definition}[section]
\newtheorem{lemma}[definition]{Lemma}
\newtheorem{theorem}[definition]{Theorem}
\newtheorem{proposition}[definition]{Proposition}
\newtheorem{corollary}[definition]{Corollary}

\newtheorem*{theorem*}{Theorem}
\newtheorem*{corollary*}{Corollary}
\newtheorem*{lemma*}{Lemma}
\newtheorem*{proposition*}{Proposition}

\theoremstyle{remark}
\newtheorem{remark}[definition]{Remark}

\hypersetup{hidelinks}

\DeclareMathOperator{\GL}{GL}

\DeclareMathOperator{\End}{End}

\DeclareMathOperator{\Mod}{Mod}
\DeclareMathOperator{\charstc}{char}

\DeclareMathOperator{\dgend}{dg-End}

\DeclareMathOperator{\per}{per}

\DeclareMathOperator{\Ann}{Ann}

\newcommand{\mb}[1]{\mathbb{#1}}

\newcommand{\derivedcat}{D}

\newcommand{\bounder}{{\derivedcat}^b}
\newcommand{\adMod}[1]{\derivedcat(#1)}

\newcommand{\adbmod}[1]{\bounder_{fg}(#1)}

\newcommand{\cx}[1]{#1^{\bullet}}
\newcommand{\gencx}{\cx{\ourgen}}

\newcommand{\padicfield}{F}
\newcommand{\resfield}{k}

\newcommand{\roi}{\mathcal{O}}

\newcommand{\indicate}{1}
\newcommand{\coeffs}{R}
\newcommand{\triv}{\mathbbm{1}}

\newcommand{\finite}[1]{#1_f}

\newcommand{\alggp}[1]{\mathrm{#1}}
\newcommand{\padicgp}{G}
\newcommand{\redgp}{\alggp{\padicgp}}
\newcommand{\gln}{\alggp{GL}_{n}}
\newcommand{\glnp}{\gln(\padicfield)}
\newcommand{\fingp}{\finite{G}}
\newcommand{\glnf}{\gln(\resfield)}

\newcommand{\iwa}{I}
\newcommand{\parhor}{J}
\newcommand{\maxcomp}{K}
\newcommand{\prop}[1]{#1^1}
\newcommand{\propiwa}{\prop{\iwa}}

\newcommand{\propmax}{\prop{\maxcomp}}
\newcommand{\weyl}{W}
\newcommand{\torus}{T}
\newcommand{\splittor}{S}

\newcommand{\phs}{\mathcal{J}}
\newcommand{\fphs}{\finite{\phs}}

\newcommand{\finbor}{\finite{\iwa}}
\newcommand{\finuni}{\finite{\propiwa}}
\newcommand{\finpara}{\finite{\parhor}}

\newcommand{\finuniopp}{\oppos{\finuni}}
\newcommand{\fintor}{\finite{\torus}}
\newcommand{\finweyl}{\finite{\weyl}}
\newcommand{\finlevi}{{\levip}_{\finpara}}
\newcommand{\finleviuni}{\unip_{\finpara}}
\newcommand{\genchars}{X_{\finpara}}
\newcommand{\genchar}{\chi_{\finpara}}

\newcommand{\unip}{U}

\newcommand{\levip}{M}

\newcommand{\oppos}[1]{\bar{#1}}

\newcommand{\symn}{\mathfrak{S}_n}

\newcommand{\hecke}{\mathsf{H}}

\newcommand{\hecgln}[1]{\hecke_{#1}(n)}
\newcommand{\hecglnco}{\hecgln{\coeffs}}

\newcommand{\schur}{\mathsf{S}}

\newcommand{\schgln}[1]{\schur_{#1}(n)}
\newcommand{\schglnco}{\schgln{\coeffs}}

\newcommand{\finschur}{\finite{\schgln{\coeffs}}}

\newcommand{\smrep}{\Mod(\padicgp)}
\newcommand{\hecglob}{\hecke(\padicgp)}
\newcommand{\block}{\mathcal{B}}
\newcommand{\uniblock}{\block_{1}(\padicgp)}
\newcommand{\anniblock}{\block_{1}'(\padicgp)}
\newcommand{\nonuniblock}{\block_{\neq 1}(\padicgp)}
\newcommand{\unihec}{\hecke_{1}(\padicgp)}
\newcommand{\annihec}{\hecke_{1}'(\padicgp)}
\newcommand{\nonunihec}{\hecke_{\neq 1}(\padicgp)}
\newcommand{\annihilator}{\mathcal{I}}
\newcommand{\finanni}{\finite{\annihilator}}
\newcommand{\finrep}{\Mod(\fingp)}
\newcommand{\hecfin}{\hecke(\fingp)}
\newcommand{\hecfinuni}{\hecke_{1}(\fingp)}
\newcommand{\hecfinanni}{\hecke_{1}'(\fingp)}
\newcommand{\hecfinonu}{\hecke_{\neq 1}(\fingp)}

\newcommand{\uniblockfin}{\block_{1}(\fingp)}
\newcommand{\anniblockfin}{\block_{1}'(\fingp)}
\newcommand{\nonunblofin}{\block_{\neq 1}(\fingp)}
\newcommand{\heccomp}{\hecke(\maxcomp)}

\newcommand{\annigen}{Q}
\newcommand{\ourgen}{V}
\newcommand{\unigen}{\Gamma}
\newcommand{\project}{P}
\newcommand{\funig}{\finite{\unigen}}
\newcommand{\fourg}{\finite{\ourgen}}
\newcommand{\fannig}{\finite{\annigen}}
\newcommand{\finproj}{\finite{\project}}

\newcommand{\funct}[1]{\mathrm{#1}}
\newcommand{\ind}{\funct{ind}}
\newcommand{\infl}{\funct{infl}}

\newcommand{\phorind}{\funct{I}}

\newcommand{\finbigschur}{\finite{\schur_{\coeffs}(N,n)}}

\newcommand{\iwadiag}{\iwa_0}

\title{The Derived $l$-Modular Unipotent Block of $p$-adic $\GL_n$}
\author{Rose Berry}

\begin{document}

\maketitle

\begin{abstract}
    For a non-Archimedean local field $\padicfield$ of residue cardinality $q=p^r$, we give an explicit classical generator $\ourgen$ for the bounded derived category $\adbmod{\unihec}$ of finitely generated unipotent representations of $\padicgp=\glnp$ over an algebraically closed field of characteristic $l\neq p$. The generator $\ourgen$ has an explicit description that is much simpler than any known progenerator in the underived setting. This generalises a previous result of the author in the case where $n=2$ and $l$ is odd dividing $q+1$, and provides a triangulated equivalence between $\adbmod{\unihec}$ and the category of perfect complexes over the dg algebra of dg endomorphisms of a projective resolution of $\ourgen$. This dg algebra can be thought of as a dg-enhanced Schur algebra. As an intermediate step, we also prove the analogous result for the case where $\padicfield$ is a finite field.
\end{abstract}

\section{Introduction}

Let $\padicfield$ be a non-Archimedean local field of residue cardinality $q=p^r$, and let $\padicgp=\glnp$. In \cite{bernstein1984centre} it was shown that the category of smooth complex representations of $\padicgp$ decomposes into blocks, and in \cite{bushnell1999types} each block was associated to a type, which gives an explicit progenerator for the block. Furthermore, the endomorphisms of this progenerator were shown to be a finite product of extended affine Hecke algebras of type A with powers of $q$ as parameters.

In particular, the unipotent block $\uniblock$, that is, the block containing the trivial representation $\triv$, has progenerator $\project=\ind_{\iwa}^{\padicgp}\triv$ for $\iwa$ an Iwahori subgroup, and $\End(P)$ is the extended affine Hecke algebra $\hecgln{\mathbb{C}}$ of type $A_{n-1}$ and parameter $q$. Hence in particular each block is equivalent to some $\block_1(H)$, where $H$ is a finite product of general linear groups over finite extensions of $\padicfield$.

The work of \cite{bernstein1984centre} holds for any reductive group, and the theory of types has since been expanded to many other cases. In particular, types have been constructed for inner forms of $\gln$ in \cite{secherre2012smooth}, for classical groups with $p\neq 2$ in \cite{miyauchi2014semisimple}, and for tamely ramified groups where $p$ does not divide the order of the Weyl group in \cite{fintzen2021construction,fintzen2021types}. In all cases these provide a progenerator for the block and an associated twisted affine Hecke algebra.

One ongoing avenue of research is to consider representations over more general coefficient rings. Let $\coeffs$ be an algebraically closed field of characteristic $l\neq p$. In \cite{vigneras1998induced} it was shown that Bernstein's block decomposition still holds for the category $\smrep$ of smooth representations of $\padicgp$ over $\coeffs$. In \cite{chinello2018blocks,dat2018equivalences} it is shown that each block is still equivalent to some $\block_1(H)$. In the banal case, that is, when $l$ does not divide $q^e-1$ for any $1\leq e\leq n$, it is expected that $\project$ is still a generator for $\uniblock$. However, in the non-banal case, $\project$ is no longer a progenerator for $\uniblock$, and the full structure of the block remains unknown.

In \cite{vigneras2003schur}, a more complicated representation $\unigen$ is constructed using Gelfand-Graev representations. Let $\annihilator=\Ann(\project)$. It is shown that the subcategory $\anniblock$ of $\uniblock$, whose objects are the representations $M$ such that $\annihilator M=0$, has progenerator $\annigen=\unigen/\annihilator\unigen$, and that the endomorphism algebra of $\annigen$ is the Schur algebra $\schglnco$, which is closely related to $\hecglnco$. Furthermore, it is shown that there is some $N$ such that $\annihilator^N\uniblock=0$, and that the Schur algebra is also the endomorphism algebra of the simpler representation $\ourgen=\bigoplus_{\parhor\in\phs} \ind_{\parhor}^{\padicgp}\triv$, where $\phs$ is a set of standard parahoric subgroups of $\padicgp$.

Building on this, in Theorems 6.8 and 6.10 of \cite{berry2024derived} the author found a description of the derived category of $\uniblock$ for $n=2$ and $l$ odd dividing $q+1$. Specifically, we found two classical generators for the bounded derived category of finitely generated unipotent representations:
\begin{theorem}[Main Theorem]\label{main theorem}
    $\adbmod{\unihec}=\langle \annigen\rangle_{\adbmod{\unihec}}=\langle \ourgen\rangle_{\adbmod{\unihec}}$.
\end{theorem}
In this paper, we prove this main theorem for arbitrary $n$ and $l$. The high-level method is similar to \cite{berry2024derived}. We first prove a finite version of the main theorem. Let $\resfield$ be the residue field of $\padicfield$, and let $\fingp=\glnf$. Analogously to the $p$-adic case, we may define $\finproj$, $\funig$, $\finanni$, $\fannig$, and $\fourg$.
\begin{theorem}[Finite Main Theorem]\label{finite main theorem}
$\adbmod{\hecfinuni}=\langle\fannig\rangle_{\adbmod{\hecfinuni}}=\langle\fourg\rangle_{\adbmod{\hecfinuni}}.$
\end{theorem}

Note that \cref{finite main theorem} is a new result for the representation theory of $\fingp$, and provides analogous applications in the finite setting to the $p$-adic applications of \cref{main theorem}. 

The first equality of \cref{main theorem} is proven in the same way as in \cite{berry2024derived} using the following finiteness results, which were already established therein in the necessarily generality.
\begin{proposition*}[\cref{everything is noetherian,schur algebra finite global dimension}] We have that
    \begin{enumerate}
        \item $\smrep$ is Noetherian.
        \item $\schglnco$ has finite global dimension.
    \end{enumerate}
\end{proposition*}

The second equality of \cref{main theorem} is proven by lifting \cref{finite main theorem} to the $p$-adic setting via parahoric induction, which is exact and so preserves classical generators, and also preserves and reflects unipotent representations. While parahoric induction manifestly sends $\fourg$ to $\ourgen$, to show that it sends $\fannig$ to $\annigen$ we must show a further lemma connecting $\finanni$ and $\annihilator$.
\begin{lemma*}[\cref{finite annihilator lives in annihilator}]
$\hecglob\otimes_{\heccomp}\finanni=\annihilator\otimes_{\heccomp}\hecfin$
\end{lemma*}
A difference in our method compared to \cite{berry2024derived} is that we show equality here instead of just left-to-right containment, meaning we no longer need to show that $\fannig$ is a direct sum of subrepresentations of $\finproj$.

However, while the sketch so far has been identical to the methods of \cite{berry2024derived} except as noted, the proofs of \cref{finite main theorem} and \cref{finite annihilator lives in annihilator} are quite different to the ones found therein for $n=2$ and $l$ odd dividing $q+1$. For \cref{finite annihilator lives in annihilator}, we make use of general properties of the Iwahori-weyl group and global Hecke algebras to avoid needing to do explicit calculations with the elements of $\finanni$, which in turn allows us to avoid the elaborate task of describing those elements explicitly.

The proof of \cref{finite main theorem} can no longer use the explicit submodule lattices for the projective indecomposable representations we had in the case of \cite{berry2024derived}, as these were obtained via cyclic defect theory and not all cases we consider have cyclic defect. Instead, we use the Dipper-James construction of irreducible $l$-modular representations of finite $\gln$. While Dipper-James theory doesn't completely describe the projective indecomposable representations, it does provide decomposition numbers, along with some further partial results. Together with the work of \cite{takeuchi1996group}, which is a finite analogue of and inspiration for \cite{vigneras2003schur}, the theory gives sufficient information about the composition series of the projective indecomposable representations to show our claim.

As in \cite{berry2024derived}, we may use \cref{main theorem} to give a derived Morita equivalence for the derived $l$-modular unipotent block. Let $\gencx$ be a projective resolution of $\ourgen$ in $\smrep$, write $\dgend$ for the dg endomorphism algebra of a complex, and write $\per$ for the perfect complexes over a dg algebra.

\begin{corollary*}[\cref{derived equivalence}]
    Let $\gencx$ be a projective resolution of $\ourgen$ in $\smrep$. There is a triangulated equivalence $\adbmod{\unihec}\simeq\per(\dgend_{\padicgp}(\gencx))$.
\end{corollary*}

The dg algebra $\dgend_{\padicgp}(\gencx)$ is a dg enriched version of the Schur algebra $\schglnco$. In particular, the latter is the zeroth cohomology of the former. We hope in future to obtain a sufficiently explicit description of this dg Schur algebra, or its perfect complexes, which would provide an $l$-modular analogue at the derived level for the description of the complex unipotent block via the well-understood representation theory of affine Hecke algebras of type $A$.

Furthermore, one might hope to extend the results of \cite{vigneras2003schur} and this paper to other blocks and other groups. It is known (\cite{minguez2014representations,kurinczuk2020cuspidal,fintzen2022tame}) that the various types discussed above can still be constructed with $l$-modular coefficients, though the representations associated to them are no longer progenerators. One might hope that these provide subcategories analogous to $\anniblock$, constructed via the annihilators of these representations, with progenerators whose endomorphisms are in some sense twisted affine Schur algebras and which classically generate the derived category of the block.

\subsection{Acknowledgements}

I would like to thank the Engineering and Physical Sciences Research Council, grant T00046, for funding my PhD studentship and thus enabling me to do this research, as well as grant EP/V061739/1 for funding the writing of the paper. I would also like to thank my PhD supervisors, Professors Shaun Stevens and Vanessa Miemietz, for their constant guidance and support. 

\section{Derived Categories and Classical Generators}

Throughout, let $\coeffs$ be an algebraically closed field with $\charstc(R)=l$.

\begin{definition}
    A dg $\coeffs$-algebra is an $\coeffs$-complex $(B,d)$ whose underlying graded $\coeffs$-module is a graded algebra over $\coeffs$, satisfying the graded Leibniz rule
    \[
    d(fg)=d(f)g+(-1)^nfd(g)
    \]
    for all $f$ of degree $n$.
\end{definition}

Let $A$ be an ordinary (ie non-dg) $\coeffs$-algebra. We may then think of $A$ as a dg algebra with all elements having degree $0$. We will consider dg algebras arising in the following way:

\begin{definition}
    Let $(\cx{M},d')$ be a complex of $A$-modules. The dg endomorphism algebra $\dgend_A(\cx{M})$ is the dg algebra whose $\coeffs$-complex has in degree $n$ the graded $A$-module homomorphisms $f:\cx{M}\rightarrow \cx{M}$ of degree $n$, with differential
    \[
        df:=d'f-(-1)^nfd'
    \]
    for all $f:\cx{M}\rightarrow \cx{M}$ of degree $n$, and with multiplication given by componentwise composition.
\end{definition}

Let $B$ be a dg $\coeffs$-algebra with differential $d$.

\begin{definition}
A dg $B$-module is an $\coeffs$-complex $(\cx{M},d')$ whose underlying graded $R$-module is a graded $B$-module, such that
    \[
    d'(fv)=(df)v+(-1)^nf(d'v)
    \]
    for all $f\in B$ of degree $n$ and all $v\in \cx{M}$.

A morphism of dg $B$-modules is a morphism of $\coeffs$-complexes that is also a morphism of graded $B$-modules.

Write $\adMod{B}$ for the (unbounded) derived category of $B$, the localisation of the category of dg $B$-modules at quasi-isomorphisms.
\end{definition}

$\adMod{B}$ is a triangulated category whose distinguished triangles are short exact sequences of dg $B$-modules which are split as graded $B$-modules.

\begin{definition}
    We say a set of objects $G$ of a triangulated category $T$ classically generates a triangulated subcategory $T'$ of $T$ if $T'$ is the smallest full triangulated subcategory of $T$ closed under isomorphisms and direct summands and containing $G$. We write this as $T'=\langle G\rangle_T$.

    The triangulated category $\per(B)$ of perfect objects in $\adMod{B}$ is $\langle B\rangle_{\adMod{B}}$.
\end{definition}

Observe that, in the case that $B$ is an ordinary algebra $A$, we have that $\per(A)$ is the full subcategory of $\adMod{A}$ consisting of objects isomorphic to finite length complexes of finitely generated projective $A$-modules.

\begin{definition}
We write $\adbmod{A}$ for the subcategory of $\adMod{A}$ consisting of objects isomorphic to finite length complexes of finitely generated $A$-modules. 
\end{definition}
This is a triangulated subcategory of $\adMod{A}$ that is closed under direct summands in $\adMod{A}$.

The following proposition is implicitly used in Theorem 6.8 of \cite{berry2024derived}, but for completeness we give a proof.

\begin{proposition}\label{progenerators are classical generators}
    Suppose $M$ is a progenerator of $\Mod(A)$. Then $M$ classically generates $\per(A)$.
\end{proposition}

\begin{proof}
    As finitely generated projective modules are precisely the direct summands of finite direct sums of $A$, we have $\per(A)=\langle A\rangle_{\per(A)}$. Furthermore, as a progenerator is finitely generated and projective, we have $M\in\per(A)$. It thus suffices to show that $A\in\langle M\rangle_{\per(A)}$. But as $M$ is a generator, $A$ is the quotient of a direct sum of copies of $M$. As $A$ is finitely generated, this direct sum may be taken to be finite, and as $A$ is projective, the quotient splits, so $A$ is a direct summand of a finite direct sum of copies of $M$.
\end{proof}

\begin{proposition}\label{triangulated equivalence}
    Let $\mathcal{T}$ be a full triangulated subcategory of $\adMod{A}$ that is closed under direct summands, let $M$ be an object in both $\Mod(A)$ and $\mathcal{T}$, such that $\langle M\rangle _{\mathcal{T}}=\mathcal{T}$, and let $\cx{M}$ be a projective resolution of $M$ in $\Mod(A)$. Then there is a triangulated equivalence
    \[
    \mathcal{T}\simeq\per(\dgend_A(\cx{M})).
    \]
\end{proposition}

\begin{proof}
    This is Theorem 6.4 of \cite{berry2024derived}.
\end{proof}

We now generalise Lemma 6.6 of \cite{berry2024derived} and its proof to a more abstract setting.

\begin{proposition}
    If $A$ is noetherian and of finite global dimension, then every finitely generated $A$-module $M$ has a finite length finitely generated projective resolution.
\end{proposition}

\begin{proof}
    As $A$ is noetherian and $M$ is finitely generated we know by \cite{rotman2009homological} Lemma 7.19 that $M$ has a finitely generated projective resolution. But as $A$ has finite global dimension, say $n$, replacing the $n$-th term with the $(n-1)$th syzygy gives, by \cite{rotman2009homological} Proposition 8.6, a finitely generated projective resolution of length $n$.
\end{proof}

\begin{proposition}
    If $\cx{M}$ is a finite length complex of $A$-modules, and each $M^i$ has a finite length finitely generated projective resolution, then $\cx{M}$ is quasi-isomorphic to a finite length complex of finitely generated projective modules.
\end{proposition}

\begin{proof}
    For each $i$, write $P^{i\bullet}$ for a choice of finite length finitely generated projective resolution of $M^i$.

    By \cite{gelfand2003homological} Lemma III.7.12, $\cx{M}$ is quasi-isomorphic to the complex $\cx{T}$ whose terms are $T^k=\oplus_{i+j=k}P^{ij}$. As $\cx{M}$ has finite length, each $T^k$ is a finite direct sum of finitely generated projective modules, and hence is finitely generated and projective. Furthermore, as $\cx{M}$ and all of the $P^{i\bullet}$ have finite length, $\cx{T}$ also has finite length.
\end{proof}

\begin{corollary}\label{perfect is finitely generated}
    If $A$ is noetherian and of finite global dimension, then $\per(A)=\adbmod{A}$.
\end{corollary}

\section{The Derived \texorpdfstring{$l$}{l}-Modular Unipotent Block of Finite \texorpdfstring{$\GL_n$}{GLn}}

Let $\resfield$ be a finite field of characteristic $p\neq l=\charstc(\coeffs)$ and cardinality $q$, and let $\redgp=\gln$. Write $\fingp=\redgp(\resfield)$ for the $\resfield$-points of $\redgp$.

We fix in $\fingp$ a choice of minimal parabolic subgroup $\finbor$ to be the upper triangular matrices. We write $\finuni$ for the unipotent radical of $\finbor$, that is, the unipotent upper triangular matrices. We fix also a choice of maximal split torus $\fintor$ in $\finbor$ to be the diagonal matrices. Let $\finuniopp$ be the unipotent radical of the opposite parabolic of $\finbor$ with respect to $\fintor$, that is, the unipotent lower triangular matrices. The Weyl group $\finweyl=N(\fintor)/\fintor$ is isomorphic to $\symn$, the symmetric group on $\left\{1,\dots,n\right\}$, and has a canonical splitting sending each permutation in $\symn$ to the corresponding permutation matrix in $N(\fintor)\subseteq\fingp$.

\begin{definition}
    We write $\triv$ for the trivial representation.

    Write $\finrep$ for the category of $\fingp$-representations over $\coeffs$.

    Write $\finproj=\ind_{\finbor}^{\fingp}\triv$.

    Let $\uniblockfin$ be the full subcategory of $\finrep$ consisting of all representations all of whose irreducible subquotients are subquotients of $\finproj$. Note that this is a direct summand of $\finrep$ (see eg \cite{vigneras2003schur} D12), and hence a direct sum of blocks. We call the blocks in this summand, as well as the representations in the summand, unipotent. Write $\nonunblofin$ for the direct sum of all non-unipotent blocks.

    We call the block containing the trivial representation $\triv$ the principal block.
\end{definition}

Observe that the principal block is unipotent, and that both $\finproj$ and $\triv$ are unipotent and finitely generated.

Let $\fphs$ be the set of parabolic subgroups of $\fingp$ containing $\finbor$. Elements of $\fphs$ are called standard parabolic subgroups, and each element is the set of block upper triangular matrices for some block decomposition. For $\finpara\in\fphs$, let $\finlevi$ be the Levi subgroup of $\finpara$ containing $\fintor$, that is, the set of block diagonal matrices with respect to the above block decomposition. Let $\finleviuni$ be the unipotent radical of the minimal parabolic subgroup of blockwise upper triangular matrices of $\finlevi$, that is, the set of blockwise unipotent upper triangular matrices. This is generated by the positive root groups of $\finlevi$ with respect to $\fintor$: each positive root group is the unipotent upper triangular matrices in $\finleviuni$ which are zero in all but one fixed non-diagonal entry. The root group is simple when said entry is on the superdiagonal. Let $\genchars$ be the set characters of $\finleviuni$ that are nontrivial on all simple root groups but trivial on all non-simple positive root groups.

\begin{definition}
    Write
    \[\funig=\bigoplus_{\finpara\in\fphs}\bigoplus_{\genchar\in\genchars}\ind_{\finpara}^{\fingp}\infl_{\finlevi}^{\finpara}\ind_{\finleviuni}^{\finlevi}\genchar.\]

    Let $\finanni$ be the annihilator of $\finproj$ in $\hecfin$. Then put $\fannig=\funig/\finanni\funig$, and put $\hecfinanni=\hecfin/\finanni$.
    
    Put $\anniblockfin$ the full subcategory of $\finrep$ consisting of representations $M$ with $\finanni M=0$.
\end{definition}

Observe that, as $\finproj$ is unipotent, $\finanni$ contains $\hecfinonu$, and so $\fannig$ is also unipotent, and $\anniblockfin$ is a subcategory of $\uniblockfin$.

\begin{definition}
    Write
    \[\fourg=\bigoplus_{\finpara\in\fphs}\ind_{\finpara}^{\fingp}\triv.\]
\end{definition}

Observe that $\fourg$ is a finite direct sum of submodules of $\finproj$, and so is unipotent and finitely generated.

\begin{proposition}
    The unipotent part of $\funig$ is a progenerator for $\uniblockfin$.
\end{proposition}

\begin{proof}
    The proof mirrors Theorem 5.13(1) and 5.10 of \cite{vigneras2003schur} (see also \cite{takeuchi1996group} for another proof). $\funig$ is by construction finitely generated and projective, and for any unipotent irreducible representation we may apply Property H1 of \cite{vigneras2003schur} to show that it has a nonzero vector invariant under a certain unipotent subgroup, which implies that it is a quotient of $\funig$ by 5.4(3) of the same source. Hence we conclude by said source's Corollary 3.7.
\end{proof}

\begin{corollary}\label{fannig progenerator}
    $\fannig$ is a progenerator of $\anniblockfin$.
\end{corollary}

We seek to establish \cref{finite main theorem}, which shall enable us to describe the $p$-adic setting. Our proof proceeds by describing $\fannig$ using the work of Dipper and James (\cite{james1986irreducible, dipper1989schur}). Recall that a partition $\lambda$ of a nonnegative integer $n$ is a non-increasing tuple $(\lambda_i)$ of positive integers with sum $n$. The dominance order on partitions is the partial order where $\lambda\geq\mu$ precisely when $\sum_{i=1}^j\lambda_i\geq\sum_{i=1}^j\mu_i$ for all $i$. We associate to each partition $\lambda$ a standard parabolic $\finpara(\lambda)$ of $\fingp$, namely the upper block triangular matrices with the $i$th block having size $\lambda_i$.

In Theorem 8.1 of \cite{james1986irreducible}, it is shown that there is a bijection from partitions $\lambda$ of $n$ to unipotent irreducible representations $D(\lambda)$. As this claim holds for any choice of $\coeffs$ algebraically closed of characteristic $l\neq p$, it is also true for an algebraically closed field $K$ of characteristic $0$. Theorem 8.1 of \cite{james1986irreducible} gives a canonical choice $S(\lambda)$ for an $l$-modular reduction of the unipotent irreducible representation over $K$ corresponding to $\lambda$.

We now introduce the finite Schur algebra, whose decomposition matrix is deeply entwined with that of $\fingp$. In later sections, we shall see a $p$-adic analogue, which we shall simply call the Schur algebra, hence the use of the qualifier `finite' for this version (perhaps it would be better to call the $p$-adic version the `Iwahori-Schur' algebra, but we have not seen this convention anywhere).

\begin{definition}
    The finite Schur algebra $\finschur$ is the algebra $\End_{\hecfin}(\fourg)$.
\end{definition}

The surprising property of $\finschur$ that makes it relevant for us is the following:

\begin{proposition}
    $\End_{\hecfin}(\fannig)$ is Morita equivalent to $\finschur$.
\end{proposition}

\begin{proof}
    This is part (a) of the theorem in the introduction of \cite{takeuchi1996group} (see also Theorem 5.8 of \cite{vigneras2003schur}).
\end{proof}

To show that $\langle\fannig\rangle_{\adbmod{\hecfinuni}}=\langle\fourg\rangle_{\adbmod{\hecfinuni}}=\adbmod{\hecfinuni}$, we proceed by showing that the first two categories contain every unipotent irreducible representation. This in fact suffices, as the next lemma shows.

\begin{definition}
 Let $\mathscr{D}$ be the set of all unipotent irreducible representations of $\fingp$. That is, $\mathscr{D}$ is the set of $D(\lambda)$ for all partitions $\lambda$ of $n$. 
\end{definition}

\begin{lemma}
    $\adbmod{\hecfinuni}=\langle \mathscr{D}\rangle_{\adbmod{\hecfinuni}}$.
\end{lemma}

\begin{proof}
    As the $D(\lambda)$ are finitely generated and unipotent, we know that $\langle \mathscr{D}\rangle_{\adbmod{\hecfinuni}}\subseteq\adbmod{\hecfinuni}$. But all finitely generated representations of $\fingp$ have finite length, and so all objects of $\adbmod{\hecfinuni}$ arise from objects in $\mathscr{D}$ via finitely many distinguished triangles.
\end{proof}

Thus it is enough to show that $\langle\fannig\rangle_{\adbmod{\hecfinuni}}$ and $\langle\fourg\rangle_{\adbmod{\hecfinuni}}$ contain $\mathscr{D}$. We first consider $\fourg$, for which we make use of the explicit structure theory of the $D(\lambda)$.

\begin{lemma}
    $\mathscr{D}\subseteq \langle\fourg\rangle_{\adbmod{\hecfinuni}}$.
\end{lemma}

\begin{proof}
    We show $D(\kappa)\in \langle\fourg\rangle_{\adbmod{\hecfinuni}}$ by decreasing induction along the dominance order for $\kappa$. First, observe that $\ind_{\finpara({\kappa})}^{\fingp}\triv$ is a summand of $\fourg$, and so $\ind_{\finpara({\kappa})}^{\fingp}\triv\in\langle\fourg\rangle_{\adbmod{\hecfinuni}}$.
    
    Next, by Theorem 7.19(iii) of \cite{james1986irreducible}, $\ind_{\finpara({\kappa})}^{\fingp}\triv$ has a composition series with all factors of the form $S(\lambda)$ with $\lambda\geq\kappa$, in which $S(\kappa)$ occurs with multiplicity $1$. But by Theorem 8.1 of \cite{james1986irreducible}, $S(\lambda)$ itself has a composition series with all factors of the form $D(\mu)$ with $\mu\geq\lambda$, in which $D(\lambda)$ occurs with multiplicity $1$. Thus $\ind_{\finpara({\kappa})}^{\fingp}\triv$ has a composition series with all factors of the form $D(\mu)$ with $\mu\geq\kappa$, in which $D(\kappa)$ occurs with multiplicity $1$.

    But by the inductive hypothesis, all $D(\mu)$ with $\mu>\kappa$ are in $\langle\fourg\rangle_{\adbmod{\hecfinuni}}$. Thus, by considering the sequence of distinguished triangles giving the composition series $\ind_{\finpara({\kappa})}^{\fingp}\triv$ in terms of $D(\mu)$, we see that $D(\kappa)\in \langle\fourg\rangle_{\adbmod{\hecfinuni}}$.
\end{proof}

To show the same for $\fannig$, we make use of the following property, which comes from deep results about $\finschur$.

\begin{lemma}
    $\finschur$ has finite global dimension.
\end{lemma}

\begin{proof}
    By Theorem 3.7.2 of \cite{cline1990integral} (see also the main theorem of \cite{du1998cells}), a family of algebras $\finbigschur$ (written $S_q(N,n,\coeffs)$ in their notation) are quasi-hereditary. By Theorem 3.6(a) of \cite{cline1990integral} any quasi-hereditary algebra over a field has finite global dimension. But by Theorem 2.24 of \cite{dipper1989schur} and Lemma 1.3 of \cite{dipper1991tensor} $\finbigschur$ and $\finschur$ are Morita equivalent whenever $N\geq n$.
\end{proof}

This allows us to conclude by a purely formal argument.

\begin{lemma}\label{finite first part}
    $\mathscr{D}\subseteq \langle\fannig\rangle_{\adbmod{\hecfinuni}}$.
\end{lemma}

\begin{proof}
    $\fannig$ is a progenerator of $\anniblockfin$, so by \cref{progenerators are classical generators} we have that $\langle\fannig\rangle_{\adbmod{\anniblockfin}}=\per(\anniblockfin)$. But, as $\anniblockfin$ is equivalent to modules over $\finschur$, and the latter has finite global dimension, and is furthermore Noetherian (as it is a finite dimensional algebra over a field), we have by \cref{perfect is finitely generated} that $\per(\anniblockfin)=\adbmod{\anniblockfin}$. Thus $\mathscr{D}\subseteq \adbmod{\anniblockfin}=\per(\anniblockfin)=\langle\fannig\rangle_{\adbmod{\anniblockfin}}$.
\end{proof}

Thus we have \cref{finite main theorem}.

\begin{theorem}\label{finite generators equivalent}
    $\langle\fannig\rangle_{\adbmod{\hecfinuni}}=\langle\fourg\rangle_{\adbmod{\hecfinuni}}=\adbmod{\hecfinuni}$.
\end{theorem}

\section{The Derived \texorpdfstring{$l$}{l}-Modular Unipotent Block of \texorpdfstring{$p$}{p}-adic \texorpdfstring{$\GL_n$}{GLn}}

Let $\padicfield$ denote a $p$-adic field, with ring of integers $\roi$, uniformiser $\varpi$, and residue field $\resfield$. Write $\padicgp=\redgp(\padicfield)$ for the $\padicfield$-points of $\redgp=\gln$, which we consider as a topological group via the topology on $\padicfield$. Let $\maxcomp=\redgp(\roi)$. This is a maximal parahoric subgroup. Let $\propmax=1+\varpi M_{n,n}(\roi)$ be its pro-$p$ radical. Then $\fingp$ is the reductive quotient of $\maxcomp$. We fix a Haar measure on $\padicgp$ with $\mu(\propmax)=1$.

Inside $\maxcomp$, we let $\iwa$ be the preimage of $\finbor$ under the quotient $\maxcomp\rightarrow\fingp$. Then $\iwa$ is an Iwahori subgroup. We let $\iwadiag$ be the intersection of $\iwa$ and the split torus $\splittor$ of diagonal matrices. Then $\iwadiag$ is the compact part of $\splittor$. The Iwahori-Weyl group $\weyl=N(\splittor)/\iwadiag$ is isomorphic to $\mb{Z}^n\rtimes\symn$, where $\symn$ acts on $\mb{Z}^n$ by permuting the entries. Furthermore, $\weyl$ has a canonical splitting by sending $\symn$ to the permutation matrices and $(i_1,\dots,i_n)\in\mb{Z}^n$ to the diagonal matrix with $(j,j)$th entry $\varpi^{i_j}$.

Many of the properties of $\fingp$ used in the previous section to establish \cref{finite main theorem} can be proven in an analogous manner for $\padicgp$. We first recall Vign\'eras's generalisation of \cite{takeuchi1996group}.

\begin{definition}
    We write $\smrep$ for the category of smooth representations of $\padicgp$ over $\coeffs$. Then $\smrep$ is isomorphic to the category of nondegenerate modules over the global Hecke algebra $\hecglob$.

    We call the block of $\smrep$ containing the trivial representation $\triv$ the unipotent block. Write $\nonuniblock$ for the direct sum of all non-unipotent blocks.
    
    Write $\unihec$ and $\nonunihec$ for the corresponding direct sums of block algebras of $\hecglob$.

        For a representation $M$, write $M_1$ for its summand in the unipotent block.

    Write $\phorind_{\fingp,\maxcomp}^{\padicgp}=\ind_{\maxcomp}^{\padicgp}\infl_{\fingp}^{\maxcomp}$ for the parahoric induction functor.

    Let $\project=\phorind_{\fingp,\maxcomp}^{\padicgp}\finproj=\ind_{\iwa}^{\padicgp}\triv$.

    Let $\annihilator$ be the annihilator in $\hecglob$ of $\project$.
    
    Let $\annihec=\hecglob/\annihilator$, and let $\anniblock$ be the category of $\hecglob$-modules annihilated by $\annihilator$, that is, the category of modules over $\annihec$.

    Let $\phs$ be the set of all parahoric subgroups containing $\iwa$ and contained in $\maxcomp$, that is, the preimages of $\finpara\in\fphs$ under the quotient $\maxcomp\rightarrow\fingp$. We call those parahoric subgroups in $\phs$ standard.

    Let $\unigen=\phorind_{\fingp,\maxcomp}^{\padicgp}\funig=\bigoplus_{\parhor\in\phs}\bigoplus_{\genchar\in\genchars}\ind_{\parhor}^{\padicgp}\infl_{\finlevi}^{\parhor}\ind_{\finleviuni}^{\finlevi}\genchar$, and let $\annigen=\unigen/\annihilator\unigen$.
\end{definition}

\begin{remark}
    In \cite{vigneras2003schur}, the definition of $\phs$ is larger: it is the set of all $\parhor$ containing $\iwa$, not just those contained in $\maxcomp$. Hence, her definition of $\unigen$ contains more summands of the form $\unigen_J=\ind_{\parhor}^{\padicgp}\infl_{\finlevi}^{\parhor}\ind_{\finleviuni}^{\finlevi}\genchar$. However, for $\gln$, all $\parhor$ containing $\iwa$ are conjugate to some $\parhor'$ containing $\iwa$ and contained in $\maxcomp$, and hence each summand $\unigen_{\parhor}$ of her $\unigen$ is isomorphic to some summand $\unigen_{\parhor'}$ of our $\unigen$. Hence, her notion of $\annigen$ is a progenerator if and only if ours is, and the two notions have Morita equivalent endomorphism algebras and classically generate the same category.
\end{remark}

By Section 5.12 of \cite{vigneras2003schur}, we have that $\project\in\uniblock$. Hence the quotient $\hecglob\rightarrow\annihec$ factors through $\unihec$, and so $\anniblock$ is a subcategory of $\uniblock$. We can now state Vign\'eras's result.

\begin{proposition}\label{Vigneras background}
    There exists some positive integer $N$ such that $\annihilator^N\uniblock=0$. Furthermore, $\anniblock$ has progenerator $\annigen$, and $\uniblock$ has progenerator $\unigen_1$.
\end{proposition}

\begin{proof}
    These are respectively Theorem 5.13 (3), Proposition 5.10, and Theorem 5.13 (1) of \cite{vigneras2003schur}.
\end{proof}

We also need several finiteness properties for $\smrep$. The first is well-known.

\begin{proposition}\label{everything is noetherian}
    $\smrep$ is Noetherian.

    In particular, when $\coeffs$ is an algebraically closed field of characteristic not $p$, $\anniblock$ and $\uniblock$ are Noetherian. The former is equivalent to the category of modules over $\schglnco$ by \cref{Vigneras background}, so $\schglnco$ is Noetherian. Similarly, the latter is equivalent to the category of modules over $\End(\Gamma_1)$, and so $\End(\Gamma_1)$ is Noetherian. But being Noetherian is preserved by Morita equivalences of non-unital rings with enough idempotents (\cite{anh1987morita}, Proposition 3.3), and so $\annihec$ and $\unihec$ are Noetherian. Thus, the ideal $\annihilator_1$ in $\unihec$ is finitely generated.
\end{proposition}

\begin{proof}
    This is \cite{dat2009finitude} Theorems 1.3 and 1.5.
\end{proof}

The second was established in the author's previous paper.

\begin{definition}
    Let $\schglnco=\End_{\padicgp}(\annigen)$.
\end{definition}

 \begin{proposition}\label{schur algebra finite global dimension}
    $\schglnco$ has finite global dimension.
 \end{proposition}

 \begin{proof}
    Theorem 3.14 of \cite{berry2024derived} shows that a certain algebra, which is also denoted therein by $\schglnco$, has finite global dimension. But by Proposition 5.8 of \cite{vigneras2003schur}, said algebra is Morita equivalent to our definition for $\schglnco$.
 \end{proof}

 We now have everything we need to prove the first equality of \cref{main theorem}, in the same manner as we did for the finite case in \cref{finite first part} and for the $\GL_2$ case in Theorem 6.8 of \cite{berry2024derived}.

\begin{lemma}\label{mod is proj}
    $\adbmod{\annihec}=\per(\annihec)$.
\end{lemma}

\begin{proof}
        By \cref{Vigneras background}, $\anniblock$ has progenerator $\annigen$. Thus, $\annihec$ is Morita equivalent to $\schglnco$. Hence it suffices to show $\adbmod{\schglnco}=\per(\schglnco)$. Now, $\schglnco$ has finite global dimension by \cref{schur algebra finite global dimension}, and it is Noetherian by \cref{everything is noetherian}. Thus we are done by \cref{perfect is finitely generated}.
\end{proof}

\begin{lemma}\label{anni to uni}
    $\adbmod{\unihec}=\langle\adbmod{\annihec}\rangle_{\adbmod{\unihec}}$
\end{lemma}

\begin{proof}
    Inclusion of the right side in the left is immediate as all $\annihec$-modules are $\unihec$-modules.
    
    Let $\cx{M}$ be an object in $\adbmod{\unihec}$. Then by \cref{Vigneras background}, we have some finite $N$ such that $\annihilator^N\cx{M}=0$. Observe that, as $\cx{M}$ is a complex of unipotent representations, we have that for any $i\geq0$, $\annihilator^i\cx{M}=\annihilator_1^i\cx{M}$. Now, by \cref{everything is noetherian}, $\annihilator_1$ is finitely generated, and $\cx{M}$ can be taken to be a complex of finitely generated modules by definition of $\adbmod{\unihec}$, so the $\annihilator^i\cx{M}$ are also complexes of finitely generated modules. Hence the quotients $\annihilator^i\cx{M}/\annihilator^{i+1}\cx{M}$ are objects in $\adbmod{\annihec}$. Thus $\cx{M}$ is a repeated extension of complexes in $\adbmod{\annihec}$, and so is in $\langle\adbmod{\annihec}\rangle_{\adbmod{\unihec}}$.
\end{proof}

\begin{theorem}\label{vigneras generator derived generates}
    $\adbmod{\unihec}=\langle \annigen\rangle_{\adbmod{\unihec}}$.
\end{theorem}

\begin{proof}
    By \cref{Vigneras background}, $\annigen$ is a progenerator for $\anniblock$. Thus by \cref{progenerators are classical generators} we have that $\per(\annihec)=\langle \annigen\rangle_{\per(\annihec)}$. Thus we are done by \cref{mod is proj} and \cref{anni to uni}.
\end{proof}

We now want to show the second equality of \cref{main theorem}. The idea is to lift \cref{finite main theorem} from $\fingp$ to $\padicgp$. To do this, we first relate $\uniblockfin$ and $\uniblock$:

\begin{proposition}\label{unipotent induces to unipotent}
    Let $\pi_f\in\finrep$. If $\pi_f\in\uniblockfin$, then $\phorind_{\fingp,\maxcomp}^{\padicgp}\pi_f\in\uniblock$. Conversely, if $\pi_f\in\nonunblofin$, then $\phorind_{\fingp,\maxcomp}^{\padicgp}\pi_f\in\nonuniblock$.
\end{proposition}

\begin{proof}
    This is \cite{vigneras2003schur}, Lemma D14 ($a_1$) and ($a_2$), noting that Conjecture~$H_3$ in said paper is stated to hold for $G=\glnp$.
\end{proof}

\begin{definition}
    Let $\ourgen=\phorind_{\fingp,\maxcomp}^{\padicgp}\fourg$. Thus, $\ourgen=\bigoplus_{\parhor\in\phs}\ind_{\parhor}^{\padicgp}\triv$.
\end{definition}

\begin{corollary}\label{our gen is in the unipotent}
    $\ourgen$ and $\phorind_{\fingp,\maxcomp}^{\padicgp}\fannig$ are both finitely generated and unipotent.
\end{corollary}

\begin{proof}
    $\ourgen$ and $\phorind_{\fingp,\maxcomp}^{\padicgp}\fannig$ are the image under parahoric induction of $\fourg$ and $\fannig$ respectively. Both $\fourg$ and $\fannig$ are unipotent and finitely generated, and parahoric induction preserves both properties.
\end{proof}

Thus both $\ourgen$ and $\phorind_{\fingp}^{\padicgp}(\fannig)$ are in $\adbmod{\hecfinuni}$. Now we are ready to lift \cref{finite main theorem} to $\padicgp$.

\begin{corollary}\label{induced finite generators are derived equivalent}
    $\langle \phorind_{\fingp}^{\padicgp}(\fannig)\rangle_{\adbmod{\unihec}}=\langle \ourgen\rangle_{\adbmod{\unihec}}$.
\end{corollary}

\begin{proof}
    Parahoric induction is exact, so this is immediate from \cref{finite main theorem}.
\end{proof}

We want to conclude by \cref{vigneras generator derived generates}, which says that $\adbmod{\unihec}=\langle\annigen\rangle_{\adbmod{\unihec}}$. Unfortunately,  $\phorind_{\fingp,\maxcomp}^{\padicgp}\fannig=\unigen/(\phorind_{\fingp,\maxcomp}^{\padicgp}\finanni\funig)$ is not a priori equal to $\annigen=\unigen/\annihilator \unigen$. It thus remains to show that $\phorind_{\fingp,\maxcomp}^{\padicgp}\fannig$ and $\annigen$ are in fact equal, that is, that $\phorind_{\fingp,\maxcomp}^{\padicgp}\finanni\funig=\annihilator \unigen$. 

To prove this equality, it is simplest to work with $\hecglob$-modules. We write $\indicate_{X}$ for the indicator function of a set $X$. We may view $\heccomp$ as the subalgebra of $\hecglob$ consisting of functions supported on $\maxcomp$. The image of the endomorphism of $\heccomp$ given by $f\mapsto \indicate_{\propmax}f\indicate_{\propmax}$ may be identified with $\hecfin$ via $g\leftrightarrow\indicate_{g\propmax}$. Viewing representations of $\fingp$, $\maxcomp$ and $\padicgp$ as modules over $\hecfin$, $\heccomp$ and $\hecglob$ respectively, parahoric induction is then just $\hecglob\otimes_{\heccomp}-$, where $\hecglob$ and $\hecfin$ are viewed as $\heccomp$-algebras via the aforementioned maps.

Using this, we may rewrite $\phorind_{\fingp,\maxcomp}^{\padicgp}\finanni\funig$ and $\annihilator\unigen$ as
\begin{align}\label{formula for induced finite generator}
    \begin{split}
        \phorind_{\fingp,\maxcomp}^{\padicgp}\finanni\funig &= \hecglob\otimes_{\heccomp}\finanni\funig
    \end{split}
\end{align}
and
\begin{align}\label{formula for generator}
    \begin{split}
        \annihilator\unigen &= \annihilator\hecglob\otimes_{\heccomp}\funig 
        =\annihilator\otimes_{\heccomp}\hecfin\funig
    \end{split}
\end{align}
respectively.

Hence, it will suffice to show that $\hecglob\otimes_{\heccomp}\finanni=\annihilator\otimes_{\heccomp}\hecfin$, as subsets of $\hecglob\otimes_{\heccomp}\hecfin$.

There are a series of simplifications that can be made to this picture. First, recall that, by the Iwahori decomposition, $\project=\ind_{\iwa}^{\padicgp}\triv$ is generated by elements of the form $\indicate_{iw\iwa}$ for $w\in\weyl$ and $i\in\iwa$, and $\hecglob$ acts on these elements by convolution on the left. Thus we may view $\project$ as a left ideal in $\hecglob$. Hence, if $Z\in\hecglob$, then $Z\in\annihilator$ if and only if, for all $i\in\iwa$ and $w\in\weyl$, we have that $Z\indicate_{iw\iwa}=0$.

Similarly, by the Bruhat decomposition, $\finproj=\ind_{\finbor}^{\fingp}\triv$ is generated by the elements $iw_f\finbor$ for $w_f\in\finweyl$ and $i\in \finbor$, and $\hecfin$ acts on these elements by left multiplication, so $\finproj$ can be analogously viewed as a left ideal in $\hecfin$. Thus, if $Z\in\hecfin$, then $Z\in\finanni$ if and only if, for all $i\in\finbor$ and $w_f\in\finweyl$, we have that $Ziw_f\finbor=0$ .

Next, observe that the map $g\mapsto\indicate_{g\propmax}$ identifies $\hecfin$ with a subalgebra of $\heccomp$, and thus of $\hecglob$. This identifies the element $iw_f\finbor\in\finproj$ with $\indicate_{iw_f\iwa}\in\hecglob$, and furthermore identifies $\finanni\subseteq\hecfin$ as a subset of $\hecglob$. Therefore, $Z\in\finanni$ if and only if, for all $i\in\finbor$ and $w_f\in\finweyl$, we have that $Z\indicate_{iw_f\iwa}=0$.

Furthermore, the above map is a splitting of the quotient map $\heccomp\rightarrow\hecfin$, and so in $\hecglob\otimes_{\heccomp}\hecfin$ the element $f\otimes g$ is equal to the element $f\indicate_{g\propmax}\otimes 1$. Hence, $\hecglob\otimes_{\heccomp}\finanni=\hecglob\finanni\otimes 1$. Now, $\annihilator$ is a left ideal in $\hecglob$, and so to show $\hecglob\finanni\otimes 1\subseteq\annihilator\otimes_{\heccomp}\hecfin$ it is enough to show that $\finanni\otimes 1\subseteq\annihilator\otimes_{\heccomp}\hecfin$ as subsets of $\hecglob\otimes_{\heccomp}\hecfin$. Thus, in particular, it suffices to show that $\finanni\subseteq\annihilator$.

Conversely, as $f\otimes g=f\indicate_{g\propmax}\otimes 1$ we also have that $\annihilator\otimes_{\heccomp}\hecfin=\annihilator\hecfin\otimes 1$, and as $\annihilator$ as is a two-sided ideal we hence have $\annihilator\otimes_{\heccomp}\hecfin=\annihilator\otimes 1=\annihilator\indicate_{\propmax}\otimes 1$. Thus, to show that $\annihilator\otimes_{\heccomp}\hecfin\subseteq \hecglob\otimes_{\heccomp}\finanni$ it is enough to show that $\annihilator\indicate_{\propmax}\subseteq \hecglob\finanni$.

\begin{lemma}\label{finite annihilator lives in annihilator}
    $\hecglob\otimes_{\heccomp}\finanni=\annihilator\otimes_{\heccomp}\hecfin$.
\end{lemma}

\begin{proof}
   By the previous remarks, to show the left-to-right inclusion, it suffices to show that for any $Z\in\finanni$, $w\in\weyl$, and $i\in\iwa$, we have $Z\indicate_{iw\iwa}=0$.

   First, observe that $Z\in\hecfin$, and so $Z$ is a linear combination of terms of the form $\indicate_{g\propmax}$. Now, as $\mu(\propmax)=1$ by our convention for normalisation, we have $Z=Z\indicate_{\propmax}$, and so $Z\indicate_{iw\iwa}=Z\indicate_{\propmax}\indicate_{iw\iwa}$.
   
   Now, by the convolution formula,
   \begin{align*}
    \indicate_{\propmax}\indicate_{iw\iwa}(x) &=\mu(\propmax\cap iw\iwa w^{-1}i^{-1})\sum_{k\in \propmax/(\propmax\cap iw\iwa w^{-1}i^{-1})}\indicate_{\propmax}(k)\indicate_{iw\iwa}(k^{-1}x) \\
    &=[\propmax:\propmax\cap iw\iwa w^{-1}i^{-1}]^{-1}\indicate_{\propmax iw\iwa}.
   \end{align*} 
 As $\propmax$ is a pro-$p$ group, $c=[\propmax:\propmax\cap iw\iwa w^{-1}i^{-1}]^{-1}$ is well-defined and nonzero in $\coeffs$, so $Z\indicate_{\propmax}\indicate_{iw\iwa}=cZ\indicate_{\propmax iw\iwa}$.

   Let $w_0$ be a minimal length coset representative for $w$ in $\finweyl\backslash\weyl$. Hence $w=w_fw_0$ for some $w_f\in\finweyl$. 

   As $w_f\in\finweyl\subseteq\maxcomp$, $w_f$ normalises $\propmax$. Similarly, as $i\in\iwa\subseteq\maxcomp$, we have that $i$ also normalises $\propmax$. Hence $Z\indicate_{\propmax iw\iwa}=Z\indicate_{iw_f\propmax w_0\iwa}$. 

   Now, by Lemma 3.19 and Variant 3.22 of \cite{morris1993tamely}, we have that $\propmax(\maxcomp\cap w_0\iwa w_0^{-1})$ is a standard parahoric subgroup, and so in particular contains $\iwa$. Multiplying by $w_0\iwa$ on the right thus gives 
   \begin{align*}
    \begin{split}
           \iwa w_0\iwa &\subseteq \propmax(\maxcomp \cap w_0\iwa w_0^{-1})w_0\iwa \\
           &\subseteq \propmax(\maxcomp w_0\iwa\cap w_0\iwa) \\
           &= \propmax w_0\iwa.
    \end{split}
   \end{align*}
     Thus we have $\propmax w_0 \iwa \subseteq \iwa w_0\iwa\subseteq \propmax w_0\iwa$, so we have equality $\iwa w_0\iwa=\propmax w_0\iwa$. Hence $Z\indicate_{iw_f\propmax w_0\iwa}=Z\indicate_{iw_f\iwa w_0\iwa}$.

    Now, we write $Z=\sum_{g\in\fingp}r_g\indicate_{g\propmax}$ for some $r_g\in\coeffs$. Recall that $i$ and $w_f$ normalise $\propmax\subseteq\iwa$, so $iw_f\iwa w_0\iwa$ is left-$\propmax$-invariant. Thus, as $\mu(\propmax)=1$, the convolution formula gives $Z\indicate_{iw_f\iwa w_0\iwa}=\sum_{g\in\fingp}r_g\indicate_{giw_f\iwa w_0\iwa}$.

    By definition, $Z\in\finanni$. But recall that, by the previous remarks, this means that $Z\indicate_{iw_f\iwa}=0$, that is, that $\sum_{g\in\fingp}r_g\indicate_{giw_f\iwa}=0$. In particular, for any fixed coset $k\iwa$ of $\iwa$ in $\maxcomp$, we have that
    \[\sum_{\mathclap{\substack{g\in\fingp \\ k\iwa=giw_f\iwa}}}r_g=0.\]
    Hence
    \begin{align*}
    \sum_{g\in\fingp}r_g\indicate_{giw_f\iwa w_0\iwa}
    &= \sum_{k\in\maxcomp/\iwa}\left(\sum_{\substack{g\in\fingp \\ k\iwa=giw_f\iwa}}r_g\right)\indicate_{k\iwa w_0\iwa} \\
    &=\sum_{k\in\maxcomp/\iwa} 0  \\
    &=0
    \end{align*}  
    
    Putting this all together, we obtain $Z\indicate_{iw\iwa}=0$. Thus we have the left-to-right inclusion.

    To show the right-to-left inclusion, by the previous remarks it suffices to show that, for any $Z\in\annihilator$, we can write $Z\indicate_{\propmax}=\sum_{j=1}^m h_j Z_j$ for $h_j\in\hecglob$ and $Z_j\in\hecfin$, such that, for all $j$, $w_f\in\finweyl$, and $i\in\iwa$, we have that $Z_j\indicate_{iw_f\iwa}=0$.

    $Z\indicate_{\propmax}$ is right-$\propmax$-invariant, and so, since elements of $\hecglob$ have compact support, we must have $Z\indicate_{\propmax}=\sum_{g\in\padicgp/\propmax}s_g\indicate_{g\propmax}$ for some $s_g\in\coeffs$, with all but finitely many $s_g$ zero.

    We may decompose this sum as $\sum_{g\in\padicgp/\propmax}s_g\indicate_{g\propmax}=\sum_{g\in\padicgp/\maxcomp}\sum_{k\in\maxcomp/\propmax}s_{gk}\indicate_{gk\propmax}$. Observe that the inner sum is finite, and as all but finitely many $s_{gk}$ are zero, the outer sum is also finite.

    But, as $\propmax$ is normal in $\maxcomp$ and $\mu(\propmax)=1$, we have $\indicate_{gk\propmax}=\indicate_{g\propmax k \propmax}=\indicate_{g\propmax}\indicate_{k\propmax}$. Hence $\sum_{g\in\padicgp/\maxcomp}\sum_{k\in\maxcomp/\propmax}s_{gk}\indicate_{gk\propmax}=\sum_{g\in\padicgp/\maxcomp}\indicate_{g\propmax}\sum_{k\in\maxcomp/\propmax}s_{gk}\indicate_{k\propmax}$.

    Let $Z_{g}=\sum_{k\in\maxcomp/\propmax}s_{gk}\indicate_{k\propmax}$. Observe then that $Z_{g}\in\hecfin$, and that $Z\indicate_{\propmax}=\sum_{g\in\padicgp/\maxcomp}\indicate_{g\propmax}Z_g$ with all but finitely many $Z_g$ zero, so the sum is finite.

    Now let $w_f\in\finweyl$ and $i\in\iwa$. Then $\indicate_{iw_f\iwa}\in\project$. Thus, as $Z\in\annihilator=\Ann(\project)$, we have that $Z\indicate_{iw_f\iwa}=0$.
    
    As $\iwa$ contains $\propmax$, we have $\indicate_{iw_f\iwa}=\indicate_{iw_f\propmax\iwa}$. As $\propmax$ is normal in $\maxcomp$, and as $\finweyl$ and $\iwa$ are subgroups of $\maxcomp$, we have $\indicate_{iw_f\propmax\iwa}=\indicate_{\propmax i w_f \iwa}$. Hence we have that $Z\indicate_{iw_f\iwa}=Z\indicate_{\propmax i w_f \iwa}$.

    Now, as $\mu(\propmax)=1$, we have that $Z\indicate_{\propmax i w_f \iwa}=Z\indicate_{\propmax}\indicate_{iw_f\iwa}=\sum_{g\in\padicgp/\maxcomp}\indicate_{g\propmax}Z_g\indicate_{iw_f\iwa}$.

    As $Z_g$ and $\indicate_{iw_f\iwa}$ are elements of $\hecfin$, so is $Z_g\indicate_{iw_f\iwa}$. Hence the support of $Z_g\indicate_{iw_f\iwa}$ is contained in $\maxcomp$, and so the support of $\indicate_{g\propmax}Z_g\indicate_{iw_f\iwa}$ is contained in $g\maxcomp$. In particular, as $g$ ranges over cosets in $\padicgp/\maxcomp$, each $\indicate_{g\propmax}Z_g\indicate_{iw_f\iwa}$ term in the sum has disjoint support. Thus, as the sum is $0$, each term $\indicate_{g\propmax}Z_g\indicate_{iw_f\iwa}$ must be zero.

    Now, $\indicate_{\propmax g^{-1}}\indicate_{g\propmax}=\mu(g\propmax g^{-1})\indicate_{\propmax}$. Furthermore, since $\propmax$ is pro-$p$, so is $g\propmax g^{-1}$, and so $\mu(g\propmax g^{-1})$ is invertible. Hence, left multiplication of $\indicate_{g\propmax}Z_g\indicate_{iw_f\iwa}$ by $\mu(g\propmax g^{-1})^{-1}\indicate_{\propmax g^{-1}}$ gives that $0=\indicate_{\propmax}Z_g\indicate_{iw_f\iwa}=Z_g\indicate_{iw_f\iwa}$. Thus we have the right-to-left inclusion.
\end{proof}

Thus we get our desired equality.

\begin{corollary}\label{annigen induced from fingen}
    $\annigen=\phorind_{\fingp,\maxcomp}^{\padicgp}\fannig$.
\end{corollary}

\begin{proof}
    By \cref{finite annihilator lives in annihilator}, together with \cref{formula for induced finite generator} and \cref{formula for generator}, we have that $\phorind_{\fingp,\maxcomp}^{\padicgp}\finanni\funig=\annihilator\unigen$. Thus $\phorind_{\fingp,\maxcomp}^{\padicgp}\fannig=\unigen/(\phorind_{\fingp,\maxcomp}^{\padicgp}\finanni\funig)=\unigen/\annihilator \unigen=\annigen$.
\end{proof}

We now have all the ingredients to prove the second equality of \cref{main theorem}.

\begin{theorem}
    $\adbmod{\unihec}=\langle \ourgen\rangle_{\adbmod{\unihec}}$
\end{theorem}

\begin{proof}
    We already know from \cref{vigneras generator derived generates} that $\adbmod{\unihec}=\langle \annigen\rangle_{\adbmod{\unihec}}$, and from \cref{induced finite generators are derived equivalent} that $\langle \phorind_{\fingp,\maxcomp}^{\padicgp}(\fannig)\rangle_{\adbmod{\unihec}}=\langle \ourgen\rangle_{\adbmod{\unihec}}$. But by \cref{annigen induced from fingen} we have $\annigen=\phorind_{\fingp,\maxcomp}^{\padicgp}\fannig$.
\end{proof}

\begin{corollary}\label{derived equivalence}
    Let $\gencx$ be a projective resolution of $\ourgen$ in $\smrep$. There is a triangulated equivalence $\adbmod{\unihec}\simeq\per(\dgend_{\padicgp}(\gencx))$.
\end{corollary}

\begin{proof}
    This follows applying \cref{triangulated equivalence} to the category $\adbmod{\unihec}$ and its classical generator $\ourgen$.
\end{proof}

\printbibliography

\end{document}